\documentclass{amsart}
\usepackage{amsmath,amsthm,amsfonts, amssymb,amscd}
\usepackage[all]{xy}

\newtheorem{theorem}{Theorem}[section]
\newtheorem{lemma}[theorem]{Lemma}%[section]
\newtheorem{definition}[theorem]{Definition}%[section]
\newtheorem{example}[theorem]{Example}%[section]
\newtheorem{remark}[theorem]{Remark}%[section]
\newtheorem{proposition}[theorem]{Proposition}

\newcommand{\Sh}{\mbox{\rm Sh}}

\begin{document}
\title[Sum of Observables on MV-Effect Algebras]{Sum of Observables on MV-Effect Algebras}
\author[Anatolij Dvure\v{c}enskij]{Anatolij Dvure\v{c}enskij$^{1,2}$}
\date{}%
\maketitle
\begin{center}  \footnote{Keywords: Effect algebra, MV-effect algebra, monotone $\sigma$-complete effect algebra, observable, sharp observable, spectral resolution, sum of observables, Olson order, $\ell$-group, semigroup

 AMS classification: 03G12 81P15, 03B50, 06C15

The paper has been supported by the grant VEGA No. 2/0069/16 SAV
 and GA\v{C}R 15-15286S. }
Mathematical Institute,  Slovak Academy of Sciences\\
\v Stef\'anikova 49, SK-814 73 Bratislava, Slovakia\\
$^2$ Depart. Algebra  Geom.,  Palack\'{y} Univer.\\
17. listopadu 12, CZ-771 46 Olomouc, Czech Republic\\

E-mail: {\tt
dvurecen@mat.savba.sk}
\end{center}

\begin{abstract}
Using a one-to-one correspondence between observables and their spectral resolutions, we introduce the sum of any two bounded observables of a $\sigma$-MV-effect algebra. This sum is commutative, associative and with  neutral element. Under the Olson order of observables, the set of bounded observables is a partially ordered semigroup, and the set of sharp observables is even a Dedekind $\sigma$-complete $\ell$-group with strong unit.
\end{abstract}

\section{Introduction}

Quantum mechanics is a very effective way of the description of the physical world. In the last decades, quantum mechanics has important applications also in quantum information and quantum computing. In a classical physical system, the measurable events fulfil axioms of Boolean algebras. However, in the quantum mechanical world, this is not the case. Therefore, Birkhoff and von Neumann \cite{BiNe} introduced quantum logics as the event structure describing quantum mechanical experiments. Nowadays we have a whole hierarchy of quantum structures as orthomodular lattices, orthomodular posets, orthoalgebras, etc. An orthodox example important for quantum mechanics is the system $\mathcal L(H)$ of closed subspaces of a real, complex or quaternionic Hilbert space $H$. They describe the so-called yes-no quantum mechanical events.

In the Nineties, Foulis and Bennett introduced in \cite{FoBe} effect algebras as partial algebras with the primary notion $+$, a partial operation, where $a+b$ means disjunction of two mutually excluding events $a$ and $b$. The most important structure is the system $\mathcal E(H)$, the system of Hermitian operators of a Hilbert space $H$ which are between the zero and the identity operator. It describes both yes-no events and fuzzy quantum events. Another important example of effect algebras are MV-effect algebras, which are mathematically equivalent to MV-algebras.

Quantum mechanical measurements are performed by observables. They are a kind $\sigma$-homomorphisms from the Borel $\sigma$-algebra $\mathcal B(\mathbb R)$ into the quantum structure. In the classical Boolean case, observables are measurable mappings, in the case $\mathcal L(H)$ they correspond to projector-valued measures, or equivalently to Hermitian or symmetric operators on $H$, and in the case $\mathcal E(H)$, they are positive operator-valued measures, see e.g. \cite{DvPu}.

In the case of a Hilbert space $H$, there is a standard order of Hermitian operators $A\le B$ iff $(A\phi,\phi)\le (B\phi,\phi)$ for each unit vector $\phi\in H$, as well as the sum of Hermitian operators $A+B$ which defines a sum of bounded observables for $\mathcal L(H)$, see e.g. \cite{Var}. Under the standard order and with respect to the sum $+$, the class $\mathcal B(H)$ of Hermitian operators is a partially ordered group which is an antilattice, \cite{Kad}, $\mathcal E(H)$ has no lattice structure, and $\mathcal L(H)$ is a complete lattice.

Olson \cite{Ols} introduced another type of ordering of Hermitian operators. However, under this order, the class $\mathcal B(H)$ is not anymore a po-group. Using his approach, in \cite{286}, we have introduced the so-called Olson order of observables on complete effect algebras.

There is a serious problem how to define the sum of bounded observables in a general case. One of the first definition of the sum of observables for quantum logics was done in \cite{Gud} and it is based on the expectation values of observables in each state.

In the present paper we introduce sum of bounded observables using a one-to-one correspondence between observables and spectral resolutions. This method will be used for observables on $\sigma$-MV-algebras. In Section 3, we show that this sum is commutative, associative with neutral element. Using the Olson order, we show in Section 4 that the class of bounded observables is a distributive lattice as well as a lattice-ordered semigroup with respect to our sum of observables. If we study only sharp observables, they form even a Dedekind $\sigma$-complete $\ell$-group with strong unit.

\section{Preliminaries}

According to \cite{FoBe}, we say that an {\it effect algebra} is a partial algebra $E = (E;+,0,1)$ with a partially defined operation $+$ and with two constant elements $0$ and $1$  such that, for all $a,b,c \in E$,
\begin{enumerate}

\item[(i)] $a+b$ is defined in $E$ if and only if $b+a$ is defined, and in
such a case $a+b = b+a;$

 \item[(ii)] $a+b$ and $(a+b)+c$ are defined if and
only if $b+c$ and $a+(b+c)$ are defined, and in such a case $(a+b)+c
= a+(b+c)$;

 \item[(iii)] for any $a \in E$, there exists a unique
element $a' \in E$ such that $a+a'=1$;

 \item[(iv)] if $a+1$ is defined in $E$, then $a=0$.
\end{enumerate}

If we define $a \le b$ if and only if there exists an element $c \in
E$ such that $a+c = b$, then $\le$ is a partial ordering on $E$, and
we write $c:=b-a$. It is clear that $a' = 1 - a$ for all $a \in E$, and if $a\le b$, then $b-a=(b'+a)'$. For more information about effect algebras, see \cite{DvPu}.

An important family of effect algebras is connected with po-groups. We note that an Abelian group $(G;+,0)$ endowed with a partial order $\le$ is a {\it partially ordered group} (po-group, in abbreviation), if for all $f,g,h\in G$, $f\le g$ yields $f+h\le g+h$. If $\le$ entails a lattice structure on $G$, $G$ is said to be a {\it lattice ordered group} ($\ell$-group, in short). An element $u\ge 0$ of an $\ell$-group $G$ is a {\it strong unit} if given $g \in G$, there is an integer $n\ge 1$ such that $g \le nu$. The couple $(G,u)$ with a fixed strong unit $u$ is said to be a {\it unital $\ell$-group}.

If $G$ is an Abelian partially ordered group written additively, choose an element $u \in G^+:=\{g \in G \colon g \ge 0\}$, and set $\Gamma(G,u):=[0,u]=\{g \in G: 0 \le g \le u\}$. Then $(\Gamma(G,u);+,0,u)$ is an effect algebra, where $+$ is the group addition of elements from $\Gamma(G,u)$  if it exists in $\Gamma(G,u)$.

A sufficient condition to be an effect algebra an interval one is the Riesz Decomposition Property, \cite{Rav}. We say that an effect algebra $E$ satisfies the Riesz Decomposition Property (RDP for short), if $a_1+a_2=b_1+b_2$ implies that there are four elements $c_{11},c_{12},c_{21},c_{22}\in E$ such that $a_1 = c_{11}+c_{12}$, $a_2= c_{21}+c_{22}$, $b_1= c_{11} + c_{21}$ and $b_2= c_{12}+c_{22}$. Equivalently, $E$ has RDP iff $a\le b+c$ implies that there are $b_1,c_1\in E$ such that $b_1\le b$, $c_1\le c$ and $a=b_1+c_1$.

An effect algebra $E$ is {\it monotone} $\sigma$-{\it complete} if, for any sequence $a_1 \le a_2\le \cdots,$ the element $a = \bigvee_n a_n$  is defined in $E$ (we write $\{a_n\}\nearrow a$). If an effect algebra is a lattice or a $\sigma$-lattice or a complete lattice, we say that $E$ is a {\it lattice effect algebra}, a $\sigma$-{\it lattice effect algebra}, and a {\it complete lattice effect algebra}, respectively.

A lattice effect algebra $E$ is an {\it MV-effect algebra} if $a\wedge b=0$ implies $a+b$ is defined in $E$ $(a,b \in E)$. If $(M;\oplus,\odot,^*,0,1)$ is an MV-algebra (for definition of MV-algebras see e.g. \cite{CDM}), we can define a partial operation $+$ on $M$ as follows $a+b$ is defined  iff $a\le b^*$ and in such a case, $a+b:=a\oplus b$. Then $(M;+,0,1)$ is an MV-effect algebra. Conversely, if $(E;+,0,1)$ is an MV-effect algebra, then we define total binary operations $\oplus $ and $\odot$ on $E$ defined by $a\oplus b:= a+(a'\wedge b)$ and $a\odot b :=(a'\oplus b')'$, $a,b \in E$, such that $(E;\oplus,',0,1)$ is an MV-algebra. Thus MV-effect algebras are equivalent to MV-algebras, \cite{DvPu}. Moreover, a lattice ordered effect algebra is an MV-algebra iff $E$ satisfies RDP. We note that every MV-effect-algebra is a distributive lattice. In addition, due to Mundici's representation theorem of MV-algebras, \cite{CDM}, for each MV-effect algebra $E$ there is a unique (up to isomorphism) a unital $\ell$-group $(G,u)$ such that $E \cong \Gamma(G,u)$.

A lattice ordered effect algebra $E$ is a {\it Boolean effect algebra} if $a+b$ exists in $E$ iff $a\wedge b =0$. Boolean effect algebras are equivalent to Boolean algebras.

If $E$ is a lattice effect algebra, then it can be covered by a system of MV-sub-effect algebras \cite{Rie}. A similar situation can happen also for $\mathcal E(H)$ nevertheless RDP fails for $\mathcal E(H)$, see \cite{Pul}.

An effect algebra $E$ is said to be an orthoalgebra if the existence of $a+a$ in $E$ implies $a=0$. Equivalently, an effect algebra is an orthoalgebra iff for each $a\in E\setminus \{0\}$, there is no non-zero element $b$ under $a$ and $a'$.

A finite system $(a_1,\ldots, a_n)$ of elements of $E$ is {\it summable} if $a:=a_1+\cdots+a_n:=\sum_{n=1}^na_i $ exists, and the element $a$ is said to be the {\it sum} of $(a_1,\ldots,a_n)$. An arbitrary system $\{a_t: t \in T\}$ of elements of $E$ is said to be {\it summable} if every finite subsystem of $\{a_t: t \in T\}$ is summable. If, in addition, there exists $a :=\bigvee\{\sum\{a_t: t \in S\}: S$ is a finite subset of $T\}$ in $E$, the element $a$ is said to be the {\it sum} of $\{a_t: t \in T\}$, and we write $a = \sum_{t \in T}a_t$.

\section{Sum of Observables}%3

Using a one-to-one relationship between observables and spectral resolutions, we introduce the sum of any two bounded observables of a $\sigma$-MV-effect algebra. We establish that the sum is commutative, associative, we find a neutral element.

Let $(\Omega,\mathcal S)$ be a measurable space, i.e. $\Omega$ is a non-void set and $\mathcal S$ is a $\sigma$-algebra of subsets of $\Omega$. Let $f,g:\Omega \to \mathbb R$ be $\mathcal S$-measurable functions. It is well-known that $f+g$ is also $\mathcal S$-measurable. The proof of this fact is based on the property
$$
\{\omega \in \Omega: f(\omega)+g(\omega)<t\}= \bigcup_{r \in \mathbb Q}(\{\omega \in \Omega: f(\omega)<r\} \cap \{\omega \in \Omega: g(\omega)<t-r\})\eqno(3.0)
$$
for each $t \in \mathbb R$, where $\mathbb Q $ is the set of rational numbers, see e.g. \cite[Thm 19.A]{Hal}.

An analogue of measurable functions for effect algebras are observables.

\begin{definition}
{\rm Let $E$ be a monotone $\sigma$-complete effect algebra. An {\it observable} on $E$ is any mapping $x:\mathcal B(\mathbb R)\to E$, where $\mathcal B(\mathbb R)$ is the Borel $\sigma$-algebra of the real line $\mathbb R$, such that (i) $x(\mathbb R)=1$, (ii) if $A,B \in \mathcal B(\mathbb R)$, $A \cap B= \emptyset$, then $x(A\cup B)=x(A)+x(B)$, and (iii) if $\{A_i\}$ is a sequence of Borel sets such that $E_i\subseteq E_{i+1}$ for each $i$ and $\bigcup_i A_i=A$, then $x(A) = \bigvee_i x(A_i)$.

An observable $x$ is {\it bounded} if there is an interval $[\alpha,\beta]$ such that $x([\alpha,\beta])=1$. We denote by $\mathcal O(E)$ and $\mathcal{BO}(E)$ the set of observables and bounded observables on $E$, respectively.
}
\end{definition}

Observables can be constructed also in the following way. Let $\{a_n\}$ be a finite or infinite sequence of summable elements, $\sum_n a_n = 1$, and let $\{t_n\}$ be a sequence of mutually different real numbers. Then the mapping $x: \mathcal B(\mathbb R) \to E$ defined by

$$
x(A):= \sum\{a_n: t_n \in A\}, \ A \in \mathcal B(\mathbb R), \eqno(3.1)
$$
is an observable on $E$. In particular, if $t_0=0$, $t_1=1$ and $a_0=a'$, $a_1=a$ for some fixed element $a \in E,$  $x$ defined by (3.1) is an observable, called the {\it question} corresponding to the element $a$, and we write $x=q_a$.

If $f: \mathbb R \to \mathbb R$ is a Borel measurable function and $x$ an observable on $E$, the mapping $f\circ x: A \mapsto x(f^{-1}(A))$, $A \in \mathcal B(\mathbb R)$, is also an observable on $E$.

In the series of papers \cite{DvKu,270,289}, there was shown that observables are in a one-to-one correspondence with spectral resolutions:

\begin{theorem}\label{th:2.2}
Let $x$ be an observable on a $\sigma$-lattice effect algebra $E$. Given a real number $t \in \mathbb R,$ we put

$$ B_x(t) := x((-\infty, t)). %\eqno(3.1)
$$
Then

$$ B_x(t) \le B_x(s) \quad {\rm if} \ t < s, \eqno (3.2)$$

$$\bigwedge_t B_x(t) = 0,\quad \bigvee_t B_x(t) =1, \eqno(3.3)
$$
and
$$ \bigvee_{t<s}B_x(t) = B_x(s), \ s \in \mathbb R. \eqno(3.4)
$$

Conversely, if there is a system $\{B(t): t \in \mathbb R\}$ of elements of $E$ satisfying {\rm
(3.2)--(3.4)}, then there is a unique observable $x$ on $E$ for which $B_x(t)=B(t)$ holds for any $t \in \mathbb R$.
\end{theorem}

The system $\{B_x(t): t \in \mathbb R\}$ from Theorem \ref{th:2.2} satisfying (3.2)--(3.4) is said to be the {\it spectral resolution} of an observable $x$. The foregoing theorem is also true if we take $\{B(t): t \in \mathbb Q\}$ satisfying (3.2)--(3.4).

We note that if $E$ is a $\sigma$-MV-effect algebra, then according to \cite[Lem 6.6.4]{CDM}, the following kind of distributive laws hold in $E$
$$
a\wedge (\bigvee_n a_n)= \bigvee_n(a\wedge a_n),\quad a\vee (\bigwedge_n a_n)= \bigwedge_n(a\vee a_n) \eqno(3.5)
$$
for each $a, a_n \in E$, $n \ge 1$.

Motivated by (3.0) and applying spectral resolutions, we have the following result.

\begin{theorem}\label{th:3.3}
Let $x,y$ be bounded observables on a $\sigma$-MV-effect algebra $E$ and let $\{B_x(t): t \in \mathbb R\}$ and $\{B_y(t): t \in \mathbb R\}$ be their spectral resolutions.
We define
$$
B_{x+y}(t):= \bigvee_{r \in \mathbb Q}(B_x(r)\wedge B_y(t-r)),\quad t \in \mathbb R. \eqno(3.6)
$$
Then $\{B_{x+y}(t): t \in \mathbb R\}$ satisfies conditions {\rm (3.2)--(3.4)} and there is a unique bounded observable $z$ on $E$ such that $B_{x+y}(t)=B_z(t)$ for each $t \in \mathbb R$.

In addition, $B_{x+y}(t)=B_{y+x}(t)$ for each $t \in \mathbb R$.
\end{theorem}

\begin{proof}
(3.2). It is evident.

(3.3). There is a bounded interval $[\alpha,\beta]$ such that $x([\alpha,\beta])=1= y([\alpha,\beta])$. Given $t\in \mathbb R$, the equation $t-r=r$ has the solution $r_0=t/2$. Choose $t$ such that $r_0=t/2 <\alpha$. If $r\le r_0$, then $B_x(t)=0$. If $r>r_0$, then $t-r < t-r_0=r_0<\alpha$, so that $B_y(t-r)=0$ which proves $B_{x+y}(t)=0$. Hence,
$\bigwedge_{t \in \mathbb R} B_{x+y}(t)=0$.

Now let $r, t-r>\beta$, then $B_x(r)=B_y(t-r)=1$ and $B_{x+y}(t)=1$, so that $\bigvee_{t\in \mathbb R} B_{x+y}(t)=1$.

(3.4). Check
\begin{eqnarray*}
\bigvee_{t<s} B_{x+y}(t)&=&\bigvee_{t<s} \bigvee_{r\in \mathbb Q}(B_x(r)\wedge B_y(t-r))\\
&=& \bigvee_{r\in \mathbb Q}\bigvee_{t<s}(B_x(r)\wedge B_y(t-r))\\
&=& \bigvee_{r \in \mathbb Q}(B_x(r)\wedge \bigvee_{t<s} B_y(t-r))\\
&=& \bigvee_{r\in \mathbb Q}(B_x(r)\wedge B_y(s-r))= B_{x+y}(s).
\end{eqnarray*}

Now we show that $B_{x+y}(t)=B_{y+x}(t)$ for each $t \in \mathbb R$. (i) Let $t$ be a rational number and set $s=t-r$. Then $\bigvee_{r \in \mathbb Q}(B_x(r)\wedge B_y(t-r))= \bigvee_{s\in \mathbb Q}(B_x(t-s)\wedge B_y(s))=B_{y+x}(t)$. If $t$ is an arbitrary real number, let $\{t_n\}$ be a sequence of real numbers such that $\{t_n\}\nearrow t$. Then using (3.4), we have  $B_{x+y}(t)=\bigvee_{t_n\nearrow t}B_{x+y}(t_n)= \bigvee_{t_n\nearrow t}B_{y+x}(t_n)= B_{y+x}(t)$.

Since the system $\{B_{x+y}(t): t \in \mathbb R\}$ satisfies (3.2)--(3.4), so that by Theorem \ref{th:2.2}, there is a unique observable $z$ on $E$ such that $\{B_{x+y}(t): t \in \mathbb R\}$ is the spectral resolution of $z$. By the proof of (3.3), we see that $z$ is a bounded observable.
\end{proof}

The bounded observable $z$ from the latter theorem is said to be the {\it sum} of  bounded observables $x$ and $y$. In addition, the subscript $x+y$ in $B_{x+y}(t)$, see (3.6), is correct because it corresponds to the spectral resolution of $z=x+y$.

Now we show that $B_{x+y}(t)$ can be calculated also in another way.

\begin{proposition}\label{pr:3.4}
Let $\mathbb S$ be a countable dense subset of real numbers. Let $x,y$ be bounded observables on a $\sigma$-MV-effect algebra $E$. Define
$$
B_{x+y}^\mathbb S(t)=\bigvee_{s \in \mathbb S}(B_x(s)\wedge B_y(t-s)),\quad t \in \mathbb R.
$$
Then $B_{x+y}^\mathbb S = B_{x+y}(t)$ for each $t \in \mathbb R$.
\end{proposition}

\begin{proof}
In the same way as in the proof of (3.4) of Theorem \ref{th:3.3}, we can prove that $\bigvee_{t<s}B_{x+y}^\mathbb S(t) = B_{x+y}^\mathbb S(t)$ for each $t \in \mathbb R$, and in addition, if  $\{t_n\}\nearrow t$, where $t_n \in \mathbb S$ for $n \ge 1$, then $\bigvee_n B_{x+y}^\mathbb S(t_n)=B_{x+y}^ \mathbb S(t)$.

For each integer $n\ge 1$ and for each $s \in \mathbb S$, there is a rational $r_n$ such that $s<r_n<s+1/n$. Therefore, $B_x(s) \wedge B_y(t-1/n-s)\le B_x(r_n)\wedge B_y(t-r_n)$ and $B_{x+y}^\mathbb S(t-1/n)\le B_x(t)$, $B_{x+y}^\mathbb S(t) = \bigvee_n B_{x+y}^\mathbb S(t-1/n) \le B_{x+y}(t)$.

In a similar way we show that $B_{x+y}(t)\le B_{x+y}^\mathbb S(t)$ for each $t \in \mathbb R$.
\end{proof}

Now we establish basic properties of the sum of observables.

\begin{proposition}\label{pr:3.5} If $x,y,z$ are bounded observable on a $\sigma$-MV-effect algebra $E$, then $x+y=y+x$ and $(x+y)+z= x+(y+z)$. If we set $o:=q_0$, then $x+o=o+x$.
\end{proposition}

\begin{proof}
The commutativity of $+$ was established in Theorem \ref{th:3.3}. For the associativity of $+$, applying (3.5), check
\begin{eqnarray*}
B_{(x+y)+z}(t)&=& B_{z+(x+y)}(t) = \bigvee_{r\in \mathbb Q}(B_z(r)\wedge B_{x+y}(t-r))\\
&=& \bigvee_{r \in \mathbb Q}\big(B_z(r)\wedge \bigvee_{s \in \mathbb Q}(B_x(s)\wedge B_y(t-r-s)\big)\\
&=& \bigvee_{r\in \mathbb Q} \bigvee_{s\in \mathbb Q}(B_z(r)\wedge B_x(s)\wedge B_y(t-r-s))\\
&=& \bigvee_{s\in \mathbb Q} \bigvee_{r\in \mathbb Q}(B_z(r)\wedge B_x(s)\wedge B_y(t-r-s))\\
&=& \bigvee_{s\in \mathbb Q}\big(B_x(s)\wedge \bigvee_{r \in \mathbb Q}(B_z(r) \wedge B_y(t-r-s)\big)\\
&=& \bigvee_{s\in \mathbb Q}(B_x(s)\wedge B_{z+y}(t-s))\\
&=& \bigvee_{s\in \mathbb Q} (B_x(s)\wedge B_{y+z}(t-s))\\
&=& B_{x+(y+z)}(t).
\end{eqnarray*}

Finally, $B_{o+x}(t)= \bigvee_{r\le 0}(B_0(r)\wedge B_x(t-r) \vee \bigvee_{r>0}(B_0(r)\wedge B_x(t-r) = 0 \vee \bigvee_{r>0}(1\wedge B_x(t-r))=B_x(t)$.
\end{proof}

\begin{example}\label{ex:3.6}
Let $q_a$ be the question observable corresponding to an element $a$ of a $\sigma$-MV-effect algebra $E$ and $n\ge 1$ be an integer. The spectral resolution of $q_a$ is
$$
B_{q_a}(t)= \left\{\begin{array}{ll} 0 & \mbox{if} \ t\le 0,\\
a' & \mbox{if}\ 0< t\le 1,\\
1 & \mbox{if}\ 1<t
\end{array}
\right.
%\eqno(3.5)
$$
for $t \in \mathbb R$. If $x_i=q_a$ for $i=1,\ldots,n$, then for $z:=x_1+\cdots+x_n$, we have $z= f_n\circ q_a$, where $f_n(t)=nt$ for each $t \in \mathbb R$. We can write $z=nq_a$. In addition, if we set $f(t)=-t$, $t \in \mathbb R$, then for $-q_a:=f\circ q_a$, we have
$$
B_{q_a+(-q_a)}(t)= \left\{\begin{array}{ll} 0 & \mbox{if} \ t\le -1,\\
a'\wedge a & \mbox{if}\ -1< t\le 0,\\
a'\vee a & \mbox{if}\ 0<t\le 1,\\
1& \mbox{if}\ 1<t,
\end{array}
\right.
%\eqno(3.5)
$$
and $q_a+(-q_a)=o$ if and only if $a\wedge a'=0$.
\end{example}

\section{Lattice and Sum Properties of Observables}%4

In the present section, we study lattice properties of bounded observables with respect to the Olson ordering. We show that under this order and the sum of observables, the set of bounded observables is a lattice-ordered semigroup and the set of bounded sharp observables is even a Dedekind $\sigma$-complete $\ell$-group with strong unit.

According to \cite{286}, we introduce on the set of observables of a $\sigma$-complete effect algebra a partial ordering, called the {\it Olson order}, $\preceq_s$ which was motivated by an ordering of Hermitian operators introduced originally by Olson \cite{Ols}. Thus, we write $x \preceq_s y$ iff $B_y(t)\le B_x(t)$ for each $t \in \mathbb R$. In \cite[Thm 3.6]{286}, there was proved that the set of bounded observables of a complete lattice effect algebra is a Dedekind complete lattice under the Olson order. Moreover, there was proved \cite[Lem 3.3]{286}:

\begin{lemma}\label{le:3.7}
Let $\{x_\alpha: \alpha \in A\}$ be a system of bounded observables on a complete lattice effect algebra $E$ such that there is a bounded observable $y$ on $E$ which is a lower bound of $\{x_\alpha: \alpha \in A\}$. Define
$$
B_x(t):= \bigvee_\alpha B_{x_\alpha}(t),\ t \in \mathbb R. \eqno(4.1)
$$
Then $\{B_x(t): t \in \mathbb R\}$ is the spectral resolution of $x = \bigwedge_\alpha x_\alpha$.
\end{lemma}

Analogously, by \cite[Lem 3.5]{286},

\begin{lemma}\label{le:3.8}
Let $\{x_\alpha: \alpha \in A\}$ be a system of bounded observables on a complete lattice effect algebra $E$ such that there is a bounded observable $y$ on $E$ which is an upper bound of $\{x_\alpha: \alpha \in A\}$. Define
$$
B_x(t):= \bigvee_{u<t} \bigwedge_\alpha B_{x_\alpha}(u),\ t \in \mathbb R. \eqno(4.2)
$$
Then $\{B_x(t): t \in \mathbb R\}$ is the spectral resolution of $x = \bigvee_\alpha x_\alpha$.
\end{lemma}

If the index set $A$ in the foregoing lemma is finite, (4.2) can be calculated simpler as it follows from the following lemma.

\begin{lemma}\label{le:3.8*}
{\rm (1)} Let $E$ be a $\sigma$-lattice effect algebra with $(3.5)$. Let $\{a_i:i \in I\}$ and $\{b_i:i \in I\}$ be two systems of elements of $E$ such that $I$ is linearly ordered and countable, and if $i,j\in I$, $i\le j$, then $a_i\le a_j$ and $b_i\le b_j$. If $a=\bigvee_i a_i$ and $b=\bigvee_i b_j$, then $\bigvee_i (a_i\wedge b_i)=a\wedge b$.

{\rm (2)} If $x$ and $y$ are two bounded observables of a $\sigma$-effect algebra $E$, then
$$
 B_{x\vee y}(t)= B_x(t)\wedge B_y(t),\quad t \in \mathbb R. \eqno(4.3)
$$
\end{lemma}

\begin{proof}
(1) We have
$$
\bigvee_i\bigvee_j (a_i\wedge b_j)= \bigvee_i(a_i\wedge \bigvee_j b_j)=
\bigvee_i (a_i\wedge b)=a\wedge b.
$$

Of course, $a_i \wedge b_i \le a\wedge b$ for each $i \in I$. Let $a_i\wedge b_i \le c$ for each $i \in I$ and  let $i,j\in I$ be arbitrary. Since $i\le j$ or $j\le i$, we have $a_i\wedge b_j \le c$ which entails $a\wedge b\le c$, and $\bigvee_i (a_i\wedge b_i)=a\wedge b$.

(2) It follows from (1).
\end{proof}

Summarizing the above facts, we show that if $E$ is a $\sigma$-MV-effect algebra, then $\mathcal{BO}(E)$ is in fact a lattice-ordered semigroup with respect to the Olson order. First we note that according to \cite{Fuc}, a semigroup $E$ with a partial ordering $\preceq$ is (i) a {\it partially ordered semigroup} if $x\preceq y$ implies $x+z\preceq y+z$ for every $x,y,z\in E$, see \cite[p. 153]{Fuc}, and (ii) a {\it lattice-ordered semigroup} if $E$ is a lattice such that $ (x\vee y)+z = (x+z)\vee (y+z)$, \cite[p. 191]{Fuc}.

\begin{theorem}\label{th:3.9}
The set $\mathcal{BO}(E)$ of bounded observables of a $\sigma$-MV-effect algebra $E$ is a distributive lattice and a lattice-ordered semigroup with respect to sum of observables and the Olson order.
\end{theorem}

\begin{proof}
(i) Due to Proposition \ref{pr:3.5}, we see that the sum of bounded observable, $+$ defined by (3.6), is commutative, associative with the neutral element $o$. According to \cite[Thm 3.6]{286}, see also Lemmas \ref{le:3.7}--\ref{le:3.8}, we see that $\mathcal{BO}(E)$ is a lattice under the Olson order $\preceq_s$.

(ii) We establish the following form of the distributive law on $\mathcal{BO}(E)$ $(x\vee y)\wedge z= (x\wedge z)\vee (y\wedge z)$ for all $x,y,z \in \mathcal{BO}(E)$.

By (4.1) and (4.2) we have $B_{x\vee y}(t)=\bigvee_{u<t}(B_x(u)\wedge B_z(u))$, $B_{(x\vee y)\wedge z}(t) = B_{x\vee y}(t)\vee B_z(t)$, and $B_{x\wedge z}(t) = B_x(t)\vee B_z(t)$, $B_{y\wedge z}(t)=B_y(t)\vee B_z(t)$. Using (4.3), we have $B_{(x\wedge z)\vee (y\wedge z)}(t)= B_{x\wedge z}(t)\wedge B_{y\wedge z}(t)= (B_x(t)\vee B_z(t))\wedge (B_y(t)\vee B_z(t))= (B_x(t)\wedge B_y(t))\vee B_z(t)= B_{x\vee y}(t) \vee B_z(t)= B_{(x\vee y)\wedge z}(t)$.

The proof the second distributivity law $(x\wedge y)\vee z = (x\vee z)\wedge (y\vee z)$ follows from the first one if we define $g(t)=-t$, $-x:=g\circ x$. Then (i) $-(-x)=x$ for each observable $x$ on $E$, (ii)  $x\preceq_s y$ iff $-y\preceq_s -x$, (iii) $-(x\vee y)=-x\wedge -y$ and $-(x\wedge y)=-x\vee -y$.

(iii) Now we show that if, for two  bounded observables $x$ and $y$ on $E$, we have $x\preceq_s y$, then $x+z\preceq_s y+z$ for each bounded observable $z$ on $E$. Indeed, if $x \preceq_s y$, then $B_y(t)\le B_x(t)$ for each $t \in \mathbb R$. Using (3.6), we have $B_{y+z}(t)=\bigvee_r (B_y(r) \wedge B_z(t-r)) \le \bigvee_r (B_x(r) \wedge B_z(t-r))= B_{x+z}(t)$, which implies $x+y\preceq_s y+z$, that is, $\mathcal{BO}(E)$ is a partially ordered semigroup.

(iv) To prove that $\mathcal{BO}(E)$ is a lattice-ordered semigroup, we have to show that $ (x\vee y)+z = (x+z)\vee (y+z)$.  Use (4.3) and check
$$
B_{(x\vee y)+z(t)}=\bigvee_{r}(B_{x\vee y}(r)\wedge B_z(t-r))\\
=\bigvee_{r}(B_x(r)\wedge B_y(r)\wedge B_z(t-r)
$$
and
\begin{eqnarray*}
B_{(x+z)\vee (y+z)}(t)&=& B_{x+y}(t)\wedge B_{y+z}(t)\\
&=&\big(\bigvee_{r }(B_x(r)\wedge B_y(t-r))\big)\wedge \big(\bigvee_{s }(B_x(s)\wedge B_y(t-s))\big)\\
&=&\bigvee_{r }\bigvee_{s } (B_x(r)\wedge B_y(t-r)\wedge B_y(s)\wedge B_z(t-s))\\
&=& \bigvee_{r}\bigvee_{s\le r } (B_x(r)\wedge B_y(t-r)\wedge B_y(s)\wedge B_z(t-s))=:B_1\\
&\vee& \bigvee_{r}\bigvee_{s>r } (B_x(r)\wedge B_y(t-r)\wedge B_y(s)\wedge B_z(t-s))=:B_2.
\end{eqnarray*}
Then
\begin{eqnarray*}
B_1&=& \bigvee_{r }\bigvee_{s \le r} (B_x(r)\wedge B_y(t-r)\wedge B_y(s))= \bigvee_r (B_x(r)\wedge B_y(t-r)\wedge B_y(r)),\\
B_2&=& \bigvee_s \bigvee_{r<s}B_x(r)\wedge B_y(s)\wedge B_z(t-s)
= \bigvee_s (B_x(s)\wedge B_y(s)\wedge B_z(t-s)).
\end{eqnarray*}
Since $B_1=B_2$, we see that $B_{(x\vee y)+z(t)}=B_{(x+z)\vee (y+z)}(t)$, that is, $\mathcal{BO}(E)$ is a lattice-ordered semigroup.
\end{proof}

We note that if $E$ is a $\sigma$-lattice effect algebra that is not an MV-effect algebra, then $x\preceq_s y$ does not implies necessarily that $x+z\preceq_s y+z$ for each bounded observable $z$. This was already mentioned for the case $\mathcal L(H)$, see \cite{Ols}.

We say that an element $a$ of an effect algebra $E$ is {\it sharp} if $ a\wedge a'$ exists in $E$ and $a\wedge a'=0$. We denote by $\Sh(E)$ the set of sharp elements of $E$. We have (i) $0,1\in \Sh(E),$ (ii) if $a \in \Sh(E)$, then $a'\in \Sh(E)$. If $E$ is a lattice effect algebra, then $\Sh(E)$ is an orthomodular lattice which is a subalgebra and a sublattice of $E$,  \cite{JeRi}. In addition, if $E$ is a $\sigma$-lattice effect algebra, then  $\Sh(E)$ is a $\sigma$-sublattice effect algebra, see \cite[Thm 4.1]{289}. If an effect algebra $E$ satisfies RDP, then by \cite[Thm 3.2]{Dvu2}, $\Sh(E)$ is even a Boolean algebra, in addition, if $E$ is a monotone $\sigma$-complete effect algebra with RDP, then $\Sh(E)$ is a Boolean $\sigma$-lattice, \cite[Thm 5.11]{Dvu2}. The latter case happens, for example, if $E$ is a $\sigma$-MV-effect algebra.

An observable $x$ is said to be {\it sharp} if $x(A) \in \Sh(E)$ for each $A \in \mathcal B(\mathbb R)$. An analogue of the following proposition was established in \cite[Prop 4.3]{289}. We denote by $\mathcal {SBO}(E)$ the set of bounded sharp observables on $E$.

\begin{proposition}\label{pr:3.10}
Let $x$ be an observable on a $\sigma$-MV-effect algebra $E$. Then $x$ is a sharp observable if and only if $B_x(t) \in \Sh(E)$ for each $t \in \mathbb R$.
\end{proposition}

\begin{proof}
One implication is evident. For the second one, suppose $B_x(t)$ is sharp for each $t \in \mathbb R$. Let $\mathcal K:=\{A \in \mathcal B(\mathbb R): x(A) \in \Sh(E)\}$. We assert that $\mathcal K$ is a Dynkin system, that is a system of subsets containing its universe which is closed under the set theoretical complements and countable unions of  disjoint subsets, \cite{Bau}. Assume $A,B \in \mathcal K$, $A\cap B = \emptyset$. Let $a \in E$ be such that $a \le x(A\cup B), x((A\cup B)^c)$. Due to RDP of $E$, there are $a_1,a_2\in E$ with $a_1 \le x(A)$ and $a_2\le x(B)$ such that $a=a_1 +a_2$. Then $a_1\le a_1+a_2\le (x(A)+x(B))'\le x(A)'$ which yields $a_1=0$, and in a similar way, we have $a_2=0$. Hence, $a=0$, $x(A\cup B) \in \Sh(E)$, and $A\cup B \in \mathcal K$.

Now let $\{A_i\}$ be a sequence of mutually disjoint Borel subsets from $\mathcal K$. By the first part of the present proof, $B_n=A_1\cup \cdots \cup A_n \in \mathcal K$. Applying \cite[Prop 4.3]{289}, we see that $A = \bigcup_n A_n = \bigcup_n B_n \in \mathcal K$.

Hence, $\mathcal K$ is clearly a Dynkin system containing all intervals of the form $(-\infty, t)$, $t \in \mathbb R$. These intervals form a $\pi$-system, i.e. intersection of any two sets from the $\pi$-system is from the $\pi$-system. Hence, by the Sierpi\'nski Theorem, \cite[Thm 1.1]{Kal}, $\mathcal K$ is a $\sigma$-algebra, which proves $\mathcal K= \mathcal B(\mathbb R)$.
\end{proof}

In what follows, we show that that $\mathcal{SBO}(E)$ is in fact an $\ell$-group under the Olson order with strong unit $q_1$.

\begin{theorem}\label{th:3.11}
The set $\mathcal{SBO}(E)$ of sharp bounded observables of a $\sigma$-MV-effect algebra $E$ is a unital $\ell$-group with strong unit $q_1$. In addition, it is a subsemigroup of the partially ordered semigroup $\mathcal{BO}(E)$ under the Olson order.
\end{theorem}

\begin{proof}
According to hypotheses, the set of sharp elements of $E$ is a Boolean $\sigma$-algebra, see \cite[Thm 5.11]{Dvu2}. If $x$ and $y$ are two bounded sharp observables, then by (3.6), every $B_{x+y}(t)$ is a sharp element and whence due to Proposition \ref{pr:3.10}, $x+y$ is a sharp bounded observable.

(i) Let $a,b $ be sharp elements of $E$ such that $a\wedge b = 0$. Then we assert that $a+ b$ is defined in $E$ and $a\vee b = a+b$. Indeed, let $c\le a,b$. There is an element $d\in E$ such that $c=b+d\ge a$. Due to RDP, $a = a_1+a_2$ where $a_1\le b$ and $a_2 \le d$. Then $a_1\le a,b$ so that $a_1=0$ and $a=a_2\le d$. This implies $a+b$ exists in $E$ in view of $c=d+b\ge a+b$. Since $a+b \ge a,b$, we see that $a+b=a\vee b$.

(ii) Let $x$ be a sharp observable on a $\sigma$-MV-effect algebra. If $A,B$ are disjoint Borel sets, then by the just proved statement, $x(A\cup B)=x(A)+x(B)=x(A)\vee x(B)$. If $A,B$ are arbitrary Borel sets, then $x(A\cup B)= x(A\setminus B) + x(A\cap B) +x(B\setminus A)=  x(A\setminus B) \vee x(A\cap B) \vee x(B\setminus A)= x(A)\vee x(B)$. Consequently, $x(\bigcup_n A_n)=\bigvee_n x(A_n)$ for each sequence $\{A_n\}$ of Borel sets.

(iii) We assert that $-x$ which is defined by $-x=f\circ x$, where $f(t)=-t$, $t\in \mathbb R$, satisfies $x+(-x)=o$ for each sharp bounded observable $x$. Indeed, $B_{x+(-x)}(t) = \bigvee_{r \in \mathbb Q}x(\{s: s<t\})\wedge x(\{s: -s<t-r\})=\bigvee_{r \in \mathbb Q} x(\{s: s<t\}\cap \{s: r-t<s\})= 0$ if $t\le 0$ and $B_{x+(-x)}(t) =1$ if $t>0$. In other words, $x+(-x)=-x+x=o$, and $-x$ is the inverse element of $x$.

Therefore, applying Theorem \ref{th:3.9}, we see that the set $\mathcal{SBO}(E)$ of sharp bounded observables on $E$ is an $\ell$-group. In addition, due to (4.1)--(4.2), it is a sublattice of $\mathcal{BO}(E)$. Of course, $\mathcal{SBO}(E)$ is a partially ordered subsemigroup of $\mathcal{BO}(E)$.

(iv) Let $x$ be a bounded observable and let $f,g$ be two bounded Borel measurable functions. Then $f\circ x + g\circ x=(f+g)\circ x$. Indeed, for each $t \in \mathbb R$, we have
\begin{eqnarray*}
B_{f\circ x+g\circ x}(t) &=& \bigvee_{r \in \mathbb Q}(B_{f\circ x}(r)\wedge B_{g\circ x}(t-r)) \\
&=&x\big(\bigcup_{r\in \mathbb Q}(\{s: f(s)<r\}\cap \{s: g(s)<t-r\})\\
&=& x(\{s: f(s)+g(s)<t\}) = B_{(f+g)\circ x}(t). \quad \mbox{see } (3.0)
\end{eqnarray*}

(v) Take the question observable $q_1$ and for each $n$, let $f_n(t)=nt$, $t \in \mathbb R$, $x_i=q_1$ for $i=1,\ldots,n$. We define $nq_a:= x_1+\cdots +x_n$. Applying (iv), we have $nq_a=f_n(q_1)$. We assert that $q_1$ is a strong unit in the $\ell$-group $\mathcal {SBO}(E)$. Indeed, let $x$ be a bounded sharp observable and let $[\alpha,\beta]$ be a closed interval such that $x([\alpha,\beta])=1$. Take an integer $n\ge 1$ such that $n>\beta$. Then $x \preceq_s nq_1$.
\end{proof}

\begin{theorem}\label{th:3.12}
Let $E$ be a $\sigma$-MV-effect algebra.
\begin{itemize}
\item[{\rm (1)}]
If $a$ and $b$ are two sharp elements such that $a+b$ exists in $E$, then $q_a+q_b=q_{a+b}$.

\item[{\rm (2)}] Let $x$ be a sharp observable such that $o\preceq_s x \preceq_s q_1$. If we set $x'=f\circ x $, where $f(t)=1-t$, then $o\preceq_s x'\preceq_s q_1$.
    In addition, $x\vee x'=q_1$ if and only if $x=q_a$ for some sharp element $a \in E$. In such a case, $a$ is unique.

\item[{\rm (3)}] If $M=\Gamma(\mathcal{SBO}(E),q_1)$, then $M$ is a $\sigma$-MV-effect algebra such that $\Sh(M)$ is $\sigma$-isomorphic to $\Sh(E)$.

\item[{\rm (4)}] The set $\mathcal{SBO}(E)$ is a Dedekind $\sigma$-complete $\ell$-group.
\end{itemize}
\end{theorem}

\begin{proof}
(1) It is evident.

(2) If $f,g$ are Borel-measurable bounded observables and $x$ is an arbitrary bounded sharp observable, it is straightforward to show that $f\circ x \vee g\circ x= h\circ x$, where $h=\max\{f,g\}$.

Hence, if $o\preceq_s x \preceq_s q_1$ and $x$ is a sharp observable, assume $x\vee x'=q_1$. Then  $1=(x\vee x')(\{1\})=x(\{s: \max\{s,1-s\}=1\})=x(\{0,1\})$. By \cite[Lem 3.1]{289}, $x=q_a$ for some element $a \in E$. Since $x$ is sharp, then so is $a$. If $x=q_b$, then $a=b$ which proves the uniqueness of $a$.

The converse implication follows from (1).

(3) By Theorem \ref{th:3.11}, the set $\mathcal{SBO}(E)$ of bounded sharp observables is an $\ell$-group and $q_1$ is a strong unit of $\mathcal{SBO}(E)$. Then $M=\Gamma(\mathcal{SBO}(E),q_1)$ is an MV-effect algebra.

Let $x\in \Sh(M)$. Then $x\vee x'=q_1$ which by (2) implies $x=q_a$ for a unique sharp element $a\in E$. Therefore, the mapping $\phi: \Sh(E)\to \Sh(M)$ defined by $\phi(a)=q_a$, $a \in \Sh(E)$, is (i) injective and surjective, (ii) $\phi(a+b)=\phi(a)+\phi(b)$ if $a+b$ exists in $E$ (see (1)), (iii) $\phi(a')=\phi(a)'$ (iv) $\phi(a\vee b)=\phi(a)\vee \phi(b)$.

(v)  If $\phi(a)+\phi(b)$ is defined in $\Sh(M)$, by (2), there is a unique sharp element $c\in E$ such that $\phi(a)+\phi(b)=\phi(c)$. Then $q_a \preceq_s q_b'=q_{b'}$ which implies $a\le b'$ and $a+b$ exists in $E$. By (1)--(2), $a+b=c$.

Therefore, $\phi$ is an isomorphism of the Boolean effect algebras $\Sh(E)$ and $\Sh(M)$.

(vi) Let $q_{a_n}\preceq_s q_{a_{n+1}}$, $n \ge 1$. Then $a_n\in \Sh(E)$, $n\ge 1$, and
$\{a_n\}\nearrow a$ for some sharp element $a$ of $E$. If for some sharp element $c\in E$, we have $q_{a_n}\preceq_s q_c$ for $n\ge 1$, then $a_n\le c$ and $a\le c$, so that $q_a\preceq_s q_c$. Hence,  $\mathcal{SBO}(E)$ is a $\sigma$-MV-effect algebra, $\phi(a)=\bigvee_n\phi(a_n)$, and  $\phi$ is a $\sigma$-isomorphism of Boolean $\sigma$-effect algebras.

(4) Due to (3)(vi), $\mathcal{SBO}(E)$ is a $\sigma$-MV-effect algebra that is in fact a Boolean $\sigma$-effect algebra. Due to \cite[Prop 16.9]{Goo}, the $\ell$-group $\mathcal{SBO}(E)$ is a Dedekind $\sigma$-complete $\ell$-group.
\end{proof}

\begin{remark}\label{re:3.13}
{\rm It is worthy of recalling that all results of Theorem \ref{th:3.3}-- Proposition \ref{pr:3.10} are valid also if $E$ is a $\sigma$-complete effect algebra for which  distributive laws (3.5) hold. Theorem \ref{th:3.11} holds for such effect algebras is in the following form: The set $\mathcal{SBO}(E)$ is a subsemigroup and a sublattice of the partially ordered semigroup $\mathcal{BO}(E)$ under the Olson order.

We note that a distributive lattice effect algebra satisfying conditions of this remark is not necessarily an MV-effect algebra, see e.g. a four-element distributive effect algebra called a diamond, \cite[Ex 1.9.23]{DvPu}.
}
\end{remark}

\end{document}